\documentclass[]{amsart}
\usepackage{amssymb,amsmath,amscd}
\usepackage{epsfig,graphicx}

\usepackage{mathrsfs}
\usepackage{amsthm}
\usepackage{latexsym}
\usepackage{enumitem}

\usepackage{stmaryrd}

\usepackage{palatino}
\usepackage[T1]{fontenc}  

\usepackage[margin=0.8in]{geometry}

\theoremstyle{plain}
\newtheorem{theorem}{Theorem}[section]
\newtheorem{corollary}[theorem]{Corollary}

\newtheorem{proposition}[theorem]{Proposition}

\newtheorem*{theorem*}{Theorem}
\newtheorem*{proposition*}{Proposition}
\theoremstyle{definition}
\newtheorem{remark}[theorem]{Remark}
\newtheorem{definition}[theorem]{Definition}

\usepackage{tabu}       

\usepackage{nicefrac}       
\usepackage{microtype}      
\usepackage{lipsum}
\usepackage[utf8]{inputenc} 
\usepackage[T1]{fontenc}    

\usepackage{graphicx}  
\usepackage{amsmath}  
\usepackage{hyperref}
\usepackage{multirow}
\usepackage{url}
\usepackage[all]{xy}
\usepackage{microtype}
\usepackage{graphicx}
\usepackage{subcaption}
\usepackage{ upgreek }
\usepackage{booktabs} 
\usepackage{ dsfont }
\usepackage{amsfonts}       
\usepackage{ amssymb }

\usepackage{amsthm}

\usepackage{forloop}
\usepackage{comment}
\usepackage{natbib}
\setcitestyle{author-year, open={(},close={)}}
\usepackage{float}
\usepackage{tabularx}

\usepackage{tikz-cd}
\usepackage{tikz}

\newcommand{\overbar}{\overline}

\newcommand{\R}{\mathbb{R}}
\newcommand{\Z}{\mathbb{Z}}

\newcommand\SComplex{K}

\newcommand{\SpeOrthogonal}{\mathrm{SO}(3)}
\newcommand{\TwoSphere}{\mathbb{S}^2}

\newcommand{\Morse}{\mathrm{Morse}(\M)}

\newcommand\Id{\mathrm{Id}}

\newcommand{\IncHomSpaceS}{\mathrm{Aut}_{\leq}(\mathbb S^1)}
\newcommand{\NonDecMapsS}{\mathrm{End}_{\leq}(\mathbb S^1)}
\newcommand\M{M}
\newcommand\bM{\partial M}
\newcommand\fbM{a}
\newcommand\manifoldn{N}

\newcommand\CyclicGroup{\Gamma_n}

\newcommand\Barc{\mathrm{Bar}}
\newcommand\persmap{\mathrm{PH}}

\newcommand\TopSpace{{X}}

\newcommand\fMorse{\mathrm{Morse}_{f}(\M)}
\newcommand\fMorseFixedPoints{\mathrm{Morse}_{f}(\M; \Crit(f),\bM)}
\newcommand\DiffM{\mathcal{D}(\M)}
\newcommand\DiffMId{\mathcal{D}_{\mathrm{Id}}(\M)}
\newcommand\DiffSId{\mathcal{D}_{\mathrm{Id}}(\mathbb{S}^1)}
\newcommand\DiffSTwoId{\mathcal{D}_{\mathrm{Id}}(\mathbb{S}^2)}
\newcommand\DiffIId{\mathcal{D}_{\mathrm{Id}}([0,1])}
\newcommand\Orb{\mathcal{O}(f)}
\newcommand\OrbId{\mathcal{O}_{\mathrm{Id}}(f)}

\newcommand\Crit{\mathrm{Crit}}

\newcommand\LevM{\mathcal{S}_{\mathrm{Id}}(f)}

\newcommand\CompFiber{\persmap^{-1}_f(D)}
\newcommand\CompFiberBis{\Omega(D)}

\newcommand\Nbhood{\mathcal{U}}

\newcommand\SufInj{f_{\CompFiberBis{}}}
\newcommand\OrbIdSufInj{\mathcal{O}_{\mathrm{Id}}(\SufInj)}
\newcommand\LevMSufInj{\mathcal{S}_{\mathrm{Id}}(\SufInj{})}
\newcommand\IntervalsSurface{\mathrm{k}_D}
\newcommand{\Emb}{\mathrm{Emb}}

\newcommand\I{\text{I}}
\newcommand\Seq{\mathrm{Seq}}
\newcommand\Val{\mathrm{Val}}

\makeatletter
\newcommand{\rmathbb}[1]{{\mathpalette\dude@rmathbb{#1}}}
\newcommand{\dude@rmathbb}[2]{%
  \scalebox{-1}[1]{$\m@th#1\mathbb{#2}$}%
  }
\makeatother

\definecolor{dred}{rgb}{0.7,0.0,0.0}

\title{Fiber of Persistent Homology on Morse functions}
\author{Jacob Leygonie}
\address{Mathematical Institute \& University of Oxford}
\email{jacob.leygonie@maths.ox.ac.uk}
\thanks{Work supported by the Mathematical Institute of Oxford \& the EPSRC grant EP/R018472/1. Both authors are members of the Centre for Topological Data Analysis.}
\author{David Beers}
\address{Mathematical Institute \& University of Oxford}
\email{David.Beers@maths.ox.ac.uk}
\begin{document}
\maketitle
\begin{abstract}
Let~$f$ be a Morse function on a smooth compact manifold~$\M$ with boundary. The path component~$\CompFiber$ containing~$f$ of the space of Morse functions giving rise to the same Persistent Homology~$D=\persmap(f)$ is shown to be the same as the orbit of~$f$ under pre-composition~$\phi \mapsto f\circ \phi$ by diffeomorphisms of~$\M$ which are isotopic to the identity. Consequently we derive topological properties of the fiber~$\CompFiber$: In particular we compute its homotopy type for many compact surfaces~$\M$. In the~$1$-dimensional settings where~$\M$ is the unit interval or the circle we extend the analysis to continuous functions and show that the fibers are made of contractible and circular components respectively. 
\end{abstract}
%
\section{Introduction}
Persistent Homology is a central and computable descriptor in Topological Data Analysis (TDA) which has been applied to a large variety of data science problems. Namely the persistence map~$\persmap$ associates to a real-valued function~$f$ on a topological space~$\TopSpace$ a so-called barcode that captures the topological variations of its sub level-sets \citep{edelsbrunner2008persistent, zomorodian2005computing}. It is natural to ask how much information can be recovered from persistent homology: Given a barcode~$D$ what can we say about the fiber~$\persmap^{-1}(D)$?  

For the purpose of this paper~$\TopSpace=\M$ is an arbitrary smooth (finite-dimensional) compact manifold with boundary~$\partial M$ and~$f:\M \rightarrow \R$ is a Morse function, i.e.~$f$ has a prescribed constant value~$\fbM_j$ on each boundary component~$\bM_j$ and has isolated critical points none of which belong to~$\bM$. Denote by~$D:=\persmap(f)$ the associated barcode. Acting on~$f$ by pre-composition with the group~$\DiffMId$ of diffeomorphisms which are isotopic to the identity, we get an orbit~$\OrbId$ inside the space of Morse functions. Our core contribution (Theorem~\ref{theorem_fiber_equal_orbit}) is the equality between this orbit and the the path connected component~$\CompFiber$ containing~$f$ in the fiber of~$\persmap$ over~$D$:
\[\CompFiber=\OrbId.\]
This result crucially relies on Mather's fibration theorem for smooth mappings~\citep{mather1969stability}, which we slightly adapt to the case of Morse functions with equal critical values using results of~\citet{cerf1970stratification}. 

We then put at work the abundant literature about the homotopy type of the orbit~$\OrbId$, especially the work of~\citet{maksymenko2006homotopy}: the mapping~$\phi \mapsto f\circ \phi$ in fact defines a locally trivial fibration from~$\DiffMId$ to the orbit~$\OrbId$, with fiber~$\LevM \subseteq \DiffMId$ the diffeomorphisms stabilising~$f$, i.e.~$f\circ \phi=f$. Hence a long exact sequence links the fiber~$\CompFiber=\OrbId$ to well-studied diffeomorphism groups of the manifold~$\M$. In particular for any compact surface~$\M$, we compute~$\pi_n(\CompFiber)$ for~$n\geq 2$ (Proposition~\ref{theorem_homotopy_type_surface_saddles_homotopy_groups}). In fact, if~$D$ has distinct interval endpoints we derive the complete homotopy type of~$\CompFiber$ for many compact surfaces (Propositions~\ref{theorem_homotopy_type_surface_no_saddles} and~\ref{theorem_homotopy_type_surface_saddles}). 

Variations of this setting have already been addressed. In the discrete setting where~$\TopSpace=\SComplex$ is a finite simplicial complex and~$f$ is compatible with face inclusions, the fiber~$\persmap^{-1}(D)$ is a complex of polyhedra~\citep{leygonie2021fiber}. This structure can be used to algorithmically reconstruct the fiber~\citep{leygonie2021algorithmic}. In the restricted case where~$\SComplex$ is a line complex, each path connected component of~$\persmap^{-1}(D)$ is contractible~\citep{cyranka2018contractibility}, and it is homeomorphic to a circle in the case where~$\SComplex$ is a subdivision of the unit circle~\citep{mischaikow2021persistent}.

In the analogous, continuous $1$-dimensional setting where~$\TopSpace$ is the interval or the circle and~$f$ is continuous, each component in the fiber is contractible and circular respectively as we show in Section~\ref{appendix_interval_circle}. For the unit interval it is possible to count the number of path connected components in~$\persmap^{-1}(D)$ by means of the combinatorics of the barcode~\citep{curry2018fiber}. For higher dimensional~$\TopSpace$ analyzing the fiber is a challenging problem: already for Morse functions on the $2$-sphere~$\TopSpace=\TwoSphere$ new tools have been designed to describe the fiber~$\persmap^{-1}(D)$, and allowed for conjectures on the number of path connected components~\citep{catanzaro2020moduli}. However the higher dimensional homotopy groups of~$\persmap^{-1}(D)$ remain unknown, from which stems the motivation of this work.
\subsection*{Acknowledgements}
We wish to thank Ulrike Tillmann for her various insights on this project. 
\section{Stability of Morse functions}
\label{sec:morse_functions}
We fix a $d$-dimensional compact smooth manifold~$\M$ with boundary~$\bM$, whose path connected components are denoted by~$\bM_j$. Given another smooth manifold~$\manifoldn$, we denote by~$C^{\infty}(\M,\manifoldn)$ the space of smooth maps from~$\M$ to~$\manifoldn$ equipped with the~$C^{\infty}$ Whitney topology. We denote by~$\DiffM\subseteq C^{\infty}(\M,\M)$ the diffeomorphisms of~$\M$, and by~$\DiffMId$ its subspace of {\em $\Id$-isotopic} diffeomorphisms, i.e. the path connected component of the identity map. 

Given real values~$\fbM_j$, a smooth map~$f$ belongs to the space~$\Morse\subseteq C^{\infty}(\M,\R)$ of {\em Morse functions} on~$\M$ if:
\begin{itemize}
\item The Hessian of~$f$ is non-degenerate at critical points, all of which belong to~$M\setminus \bM$; and
\item The restrictions $f_{|\bM_j}$ to each boundary component~$\bM_j$ are constant with prescribed value~$\fbM_j$.
\end{itemize}
Then~$\DiffM$ acts on~$C^{\infty}(\M,\R)$ by pre-composition and we denote by~$\Orb{}\subseteq \Morse$ the orbit of~$f$, and by~$\OrbId{}\subseteq \Orb{}$ the orbit of~$f$ under the restricted action of $\Id$-isotopic diffeomorphisms. 

To express the local triviality results of this section, we rely on the notion of local cross-sections defined below. 

\newcommand{\Group}{\mathcal{G}}
\newcommand{\LocalSection}{s}
\begin{definition}
\label{definition_local_cross_section}
Let~$\Group$ be a topological group acting on a topological space~$\TopSpace$. Given~$x_0\in \TopSpace$, a {\em local cross-section} for the action of~$\Group$ on~$\TopSpace$ at~$x_0$ is a continuous map~$\LocalSection:\Nbhood\rightarrow \Group$ defined on an open neighborhood~$\Nbhood$ of~$x_0$ satisfying:
\[\forall x\in \Nbhood, \, \LocalSection(x) \cdot x_0=x.\]
We say that the action of~$\Group$ on~$\TopSpace$ {\em admits local cross-sections} if it does so at each point.
\end{definition}
Up to replacing~$\LocalSection(x)$ by~$\LocalSection(x)  \cdot \LocalSection(x_0)^{-1}$ in the above definition, we can assume that~$\LocalSection(x_0)$ is the identity element. 
\begin{remark}
\label{remark_local_sections_implies_triviality}
It is well-known that, if~$\TopSpace$ admits local cross-sections, then any~$\Group$-equivariant map from a~$\Group$-space to~$\TopSpace$ is locally trivial, see e.g.~\citet[Theorem A]{palais1960local}.
\end{remark}
The main result of this section adapts Mather's stability of smooth mappings~\citep{mather1969stability} to the case of Morse functions with equal critical points by combining results of Cerf: 
\begin{proposition}
\label{prop:stability}
Let~$f\in \Morse$ and~$\fMorse$ be the subspace of Morse functions with the same critical values as~$f$. Then the action of~$\DiffMId$ on~$\fMorse$ admits local cross-sections. 
%
%
%
\end{proposition}
The first result of Cerf we use is essentially a version of Proposition~\ref{prop:stability} restricted to the space~$\fMorseFixedPoints$ of Morse functions with the same critical points and critical values as~$f$ and the same value and derivatives of any order as~$f$ on~$\bM$.
\begin{proposition}[\cite{cerf1970stratification}, Appendix, §1, Proposition~1]
\label{prop:cerf_stability_same_crit_points}
Let $f:\M \to \R$ be a Morse function. Let~$\Group$ be the subspace of diffeomorphisms fixing the critical points of~$f$ and~$\bM$, and which have the same value and derivatives of every order as the identity on~$\bM$. Then the action of~$\Group$ on~$\fMorseFixedPoints$ admits local cross-sections.
\end{proposition}
To enhance this result to Morse functions with critical points and derivatives at the boundary allowed to vary, we need the reformulation of a result of Cerf~\cite[Theorem 5]{cerf1961topologie} for the $C^\infty$ case given in \cite{hong2012diffeomorphisms}.
\newcommand{\submanifold}{N}
\newcommand{\bNbhood}{S}
\newcommand{\compactNbhood}{L}
\begin{proposition}
\label{proposition_local_triviality}
Let~$\submanifold\subseteq \M$ be a compact submanifold. Suppose that:
\begin{itemize}
    \item[(a)] either~$\submanifold$ has no boundary and does not intersect the boundary~$\bM$;
    \item[(b)] or~$\submanifold$ is a closed collar neighborhood of a boundary component~$\bM_j$.
\end{itemize}
Let~$\compactNbhood$ be a compact neighborhood of~$\submanifold$. Let~$\Emb_{ \compactNbhood}(\submanifold,\M)$ be the space of embeddings~$j:\submanifold \hookrightarrow \M$ such that~$j(\submanifold)\subseteq \compactNbhood$,~$j^{-1}(\bM)=\submanifold\cap \bM$, and that restricts to the identity on~$\submanifold\cap \bM$.
Denote by~$\mathcal{D}_{\compactNbhood}(\M)$ the space of diffeomorphisms inducing the identity on~$(\M\setminus \compactNbhood)$.
Then the action of~$\mathcal{D}_{ \compactNbhood}(\M)$ on~$\Emb_{ \compactNbhood}(\submanifold,\M)$ admits local cross-sections. 
\end{proposition}
\begin{proof}
This is a direct consequence of Theorem~3.1 in~\cite{hong2012diffeomorphisms}. Note that, strictly speaking, embeddings~$j$ as defined in~\cite[Definition~2.5]{hong2012diffeomorphisms} are required to admit an extension to a diffeomorphism of~$\M$. However this assumption is unnecessary for their Theorem~3.1 and so we omit it.
\end{proof}
\begin{proof}[Proof of Proposition~\ref{prop:stability}]
Since~$g\in \fMorse$ implies that~$\mathrm{Morse}_{g}(\M)=\fMorse$, it is enough to construct a local cross-section at~$f$. By Proposition~\ref{prop:cerf_stability_same_crit_points}, there exists a local cross-section~$g\in \Nbhood_f \mapsto \phi_g \in \DiffMId$ defined on a neighborhood~$\Nbhood_f$ of~$f$ in~$\fMorseFixedPoints$, the space of functions with the same critical points~$p_1,\cdots,p_n$ and critical values as~$f$, and with the same value and derivatives~$\partial^k f$ as~$f$ on the boundary~$\bM$. Therefore
\begin{equation}
\label{eq_1_stability_proof}
\forall g\in \Nbhood_f, \, \,  g= f \circ \phi_g. 
\end{equation}
For the general case where critical points and derivatives on the boundary are allowed to vary we simply find a diffeomorphism sending them back to~$p_1,\cdots,p_n$ and~$\partial^k f$ and apply the above result. 

Namely, from item (a) of Proposition~\ref{proposition_local_triviality}, we can find disjoint compact neighborhoods~$U_1,\cdots,U_n$ of~$p_1,\cdots,p_n$ and continuously associate to $(p'_1,\cdots,p'_n)\in U_1\times \cdots\times U_n$ a diffeomorphism~$\psi_{(p'_1,\cdots,p'_n)}\in \DiffMId$ such that~$\psi_{(p_1,\cdots,p_n)}=\Id{}$ and:
\[ \forall (p'_1,\cdots,p'_n)\in U_1\times \cdots\times U_n, \, \, \, \forall 1\leqslant i \leqslant n, \,\,\, \psi_{(p'_1,\cdots,p'_n)}(p_i)= p'_i. \]
Let~$\Nbhood\subseteq \fMorse$ be a neighborhood of~$f$ for which any~$g\in \Nbhood$ has critical points~$\Crit(g)$ in~$U_1,\cdots,U_n$. In particular in this case~$g\circ \psi_{\Crit(g)}$ has the same critical points~$p_1,\cdots,p_n$ as~$f$, so it remains to deal with the boundary~$\bM$.

Let~$\bM_j$ be a boundary component. By flowing along the normalized gradient of~$f$ (or its inverse) from the boundary~$\bM_j$ we get a compact collar~$V_j\cong \bM_j \times [0,\alpha]$ that is adapted to~$f$ in the sense that~$f(x,t)=\fbM_j\pm t$, w.l.o.g.~$f(x,t)=\fbM_j+t$. For~$g$ in a small neighborhood~$\mathcal{V}\subseteq \fMorse$ of~$f$, we have~$g(\bM_j \times [0,\frac{\alpha}{2}] )\subseteq [\fbM_j,\fbM_j+\alpha]$. Therefore, after potentially shrinking $\mathcal{V}$, we can continuously associate to~$g\in \mathcal{V}$ the embedding~$\iota_g: (x,t)\in \bM_j \times [0,\frac{\alpha}{2}] \mapsto (x,g(x,t)-\fbM_j)\in V_j$ that preserves~$\bM_j$ and satisfies~$g=f\circ \iota_g$ on~$\bM_j \times [0,\frac{\alpha}{2}]$. Hence by using item (b) of Proposition~\ref{proposition_local_triviality}, up to shrinking~$\mathcal{V}$, we can extend~$\iota_g$ to a diffeomorphism~$\chi_g$ that induces the identity outside~$V_j$. By repeating this process for each boundary component~$\bM_j$, we can continuously associate to~$g\in \mathcal{V}$ a diffeomorphism~$\chi_g$ such that~$g$ and~$f\circ \chi_g$ agrees on a closed collar neighborhood of the boundary.
 
By reducing the neighorhoods~$U_i$ and~$V_j$ to avoid overlaps, we have that for any~$g$ in~$\Nbhood\cap \mathcal{V}$ the Morse function~$g\circ \psi_{\Crit(g)} \circ \chi_{g\circ \psi_{\Crit(g)}}^{-1}$ has the same critical points as~$f$ and agrees with~$f$ on a neighborhood of the boundary~$\bM$, in particular it belongs to~$\fMorseFixedPoints$. Hence by Eq.~\eqref{eq_1_stability_proof}, up to shrinking~$\Nbhood\cap \mathcal{V}$, we have
\[\forall g\in \Nbhood\cap \mathcal{V}, g=(g\circ \psi_{\Crit(g)} \circ \chi_{g\circ \psi_{\Crit(g)}}^{-1} ) \circ \chi_{g\circ \psi_{\Crit(g)}}\circ \psi_{\Crit(g)}^{-1} =  f \circ \phi_{g\circ \psi_{\Crit(g)} \circ \chi_{g\circ \psi_{\Crit(g)}}^{-1} } \circ \chi_{g\circ \psi_{\Crit(g)}} \circ \psi_{\Crit(g)}^{-1},\]
hence the local cross-section:
\[\LocalSection: g \in \Nbhood\cap \mathcal{V} \longmapsto \phi_{g\circ \psi_{\Crit(g)} \circ \chi_{g\circ \psi_{\Crit(g)}}^{-1} } \circ \chi_{g\circ \psi_{\Crit(g)}} \circ \psi_{\Crit(g)}^{-1} \in \DiffMId. \qedhere\]
\end{proof}
\begin{corollary}
\label{corollary_path_of_functions_same_critical_values}
Let~$(f_t)_{t\in [0,1]}\subseteq \Morse$ be a path of Morse functions with the same critical values. Then there exists~$\phi \in \DiffMId$ an $\Id$-isotopic diffeomorphism such that $f_1= f_0 \circ \phi$. 
\end{corollary}
\begin{proof}
By Proposition~\ref{prop:stability}, each~$t\in[0,1]$ has a neighborhood~$\I_t\subseteq [0,1]$ such that~$f_h$ can be written~$f_h=f_t\circ \phi_h$ whenever~$h\in \I_t$. By compactness~$[0,1]$ is covered by finitely many such intervals. Therefore~$f_1$ equals~$f_0\circ \phi$, where~$\phi$ is a finite composition of diffeomorphisms in~$\DiffMId$ hence is itself in~$\DiffMId$.
\end{proof}
\section{Covering the fiber with diffeomorphisms isotopic to the identity}
\label{sec:morse_fiber}
Given~$f\in \Morse$, we get a nested sequence of sub level-sets~$f^{-1}((-\infty,x])$. In turn, by applying homology in degree~$0\leq k \leq d$ with coefficients in an arbitrary field, we get the {\em persistent homology module of~$f$}: the sequence of vector spaces~$\mathrm{H}_k(f^{-1}((-\infty,x]))$ with linear maps between them induced by inclusions, in other words a functor from the poset~$(\R,\leq)$ to finite dimensional vector spaces. The {\em barcode} of~$f$ in degree~$k$ is the isomorphism class of this functor up to natural isomorphism.
From~\citet{crawley2015decomposition} any such functor uniquely decomposes as a direct sum of functors~$\bigoplus_{(b,d)\in D} \mathbb{I}_{[b,d)}$, with $[b,d)\subseteq \R$ an interval closed on the left and open on the right (hence possibly~$d=+\infty$): each~$\mathbb{I}_{[b,d)}$ consists of $1$-dimensional vector spaces linked with identity maps on~$[b,d)$, and it is the zero vector space everywhere outside of~$[b,d)$. Therefore the barcode of~$f$, denoted by~$\persmap_k(f)$, can be equivalently described as the multi-set~$D$ of pairs~$(b,d)$ indexing this decomposition, and will be described in this way in the rest of this document. By abuse of terminology we refer to pairs~$(b,d)$ as {\em intervals} or {\em bars} of the barcode~$D=\persmap_k(f)$. Intuitively~$(b,d)$ corresponds to the appeareance of a $k$-cycle in~$f^{-1}((-\infty,b])$ that is further cancelled in~$f^{-1}((-\infty,d])$ (or persists forever if~$d=\infty$). We refer to~\citet{edelsbrunner2008persistent,zomorodian2005computing} for extensive treatments of the theory of Persistence.

In this work the persistence map is defined on Morse functions and returns the~$d+1$ barcodes of interest:
\[\persmap: f\in \Morse \longmapsto [\persmap_0(f),\cdots, \persmap_d(f)]\in \Barc^{d+1}.\]
We assume that~$\Barc$ is equipped with its natural {\em bottleneck metric} which turns~$\persmap$ into a continuous map by the Stability Theorem~\citep{cohen2007stability}. Given a barcode~$D$ and a Morse function~$f\in \Morse$ such that~$\persmap(f)=D$, we denote by~$\CompFiber$ the path connected component of the fiber~$\persmap^{-1}(D)\subseteq \Morse$ containing~$f$.
\begin{theorem}
\label{theorem_fiber_equal_orbit}
Let~$D$ be a barcode and~$f\in \persmap^{-1}(D)$. Then~$\CompFiber= \OrbId{}$.
\end{theorem}
\begin{proof}
Let~$(\phi_t)_{0\leq t \leq 1}$ be a path in~$\DiffMId$. Each~$\phi_t$ restricts to a homeomorphism between the sub level-sets of~$f\circ \phi_t$ and~$f$, hence it induces an isomorphism between the associated persistent homology modules. In turn~$\persmap(f\circ\phi_t)=\persmap(f)$, so that~$(f\circ\phi_t)_{0\leq t \leq 1}$ is a path in the fiber~$\CompFiber$, which implies~$\OrbId{}\subseteq \CompFiber$. 

Conversely let~$g\in \CompFiber$ and let~$(f_t)_{0\leq t \leq 1}$ be a path in the fiber~$\CompFiber$ joining~$f$ to~$g$, thus~$\persmap(f_t)=D$ for each~$t$. As is well-known, when~$\M$ has no boundary there is a one-to-one correspondence between the set~$\mathcal{D}$ of (bounded) interval endpoints in the barcode and the set~$\mathcal{C}$ of critical values (counted with multiplicity) for Morse functions because the associated persistent homology module and Morse-Smale complex are isomorphic~\citep{barannikov1994framed}, see also~\citep[Proposition~2.14]{leygonie2021framework} for a self-contained proof.

When~$\M$ has a boundary the correspondence adapts by adding in~$\mathcal{C}$ the value~$\fbM_j$ with multiplicity~$\sum_i \beta_i(\bM_j)$ for each boundary component~$\bM_j$ that is a local minimum. Note that a Morse function is constant on~$\bM_j$ and has no critical points there, so either it has~$\bM_j$ as a local minimum or as a local maximum, and this choice is fixed inside a path connected component of~$\Morse$. 

Therefore each~$f_t$ has the same critical values as~$f$, because the barcode~$\persmap(f_t)=D$ is constant. By corollary~\ref{corollary_path_of_functions_same_critical_values} there exists an $\Id$-isotopic diffeomorphism~$\phi$ such that~$g=f\circ \phi$. Consequently~$\CompFiber \subseteq \OrbId{}$.
\end{proof}

\section{Topological properties of the fiber}

We derive direct consequences of Theorem~\ref{theorem_fiber_equal_orbit} combined with the extensive study of~$\OrbId{}$ by~\citet{maksymenko2006homotopy}. 
Strictly speaking, it is~$\mathcal{O}_f(f)$, the path component of~$\Orb{}$~containing $f$, whose properties are studied in~\citet{maksymenko2006homotopy}. However, there is an obvious inclusion $\OrbId{}\subseteq \mathcal{O}_f(f)$, and the reverse inclusion holds as well by corollary~\ref{corollary_path_of_functions_same_critical_values}. Therefore~$\mathcal{O}_f(f)=\OrbId{}$.

Denote by~$\LevM{}$ the subspace of $\Id$-isotopic diffeomorphisms~$\phi$ preserving a Morse function~$f$, i.e.~$f\circ \phi=f$.\footnote{In~\citep{maksymenko2006homotopy} the notation~$\LevM{}$ rather stands for the space of diffeomorphisms~$\phi$ preserving~$f$ that are isotopic to~$\mathrm{Id}_{\M}$  \emph{though} maps preserving~$f$, thus it is the path connected component of~$\mathrm{Id}_{\M}$ in our~$\LevM{}$.}
\begin{proposition}
\label{theorem_fibration_orbit_stabiliser_diffeos}
Assume that~$\M$ is connected. Let~$D$ be a barcode and~$f\in \persmap^{-1}(D)$. Then the action of~$\DiffMId$ on~$\CompFiber$ defines a locally trivial principal~$\LevM$-fibration.
\end{proposition}
\begin{proof}
From~\citet[Theorem~2.1, (2)]{maksymenko2006homotopy} the action of diffeomorphisms~$\DiffM$ on~$\Orb$ defines a locally trivial principal fibration with fiber the diffeomorphisms~$\phi$ satisfying~$f\circ \phi =f$. Restricting to the action of~$\DiffMId$ on~$\OrbId$ defines a locally trivial principal~$\LevM$-fibration, and~$\CompFiber$ equals~$\OrbId{}$ by Theorem~\ref{theorem_fiber_equal_orbit}.
\end{proof}
\begin{remark}
\label{remark_LES}
The principal bundle~$\LevM\rightarrow \DiffMId\rightarrow \CompFiber$ has computationally useful and direct implications. First, it is a locally trivial fibration hence it induces a homotopy long exact sequence:
\begin{alignat*}{2}
\cdots & \to \pi_{n}(\LevM)\to \pi_{n}(\DiffMId) & \to \pi_{n}(\CompFiber)\to \pi_{{n-1}}(\LevM) \to\cdots  \to \pi_{0}(\DiffMId).         
\end{alignat*}
Second, we have the homeomorphism: 
\[\CompFiber\cong \DiffMId/\LevM.\]
\end{remark}
We apply this result to compute the path components of the fiber~$\persmap^{-1}(D)$ when~$\M$ is a circle:
\begin{proposition}
\label{prop_circle_morse}
Assume~$\M=\mathbb{S}^1$. Let~$D$ be a barcode and~$f\in \persmap^{-1}(D)$. Then~$\CompFiber$ is homotopy equivalent to~$\mathbb S^1$.  
\end{proposition}
\begin{proof}
From Proposition~\ref{theorem_fibration_orbit_stabiliser_diffeos}~$\CompFiber$ is homeomorphic to~$\DiffSId/ \LevM$. Let~$n$ be the number of minima of~$f$, which is then also the number of maxima of~$f$ because~$\chi(\mathbb S^1)=0$. Without loss of generality we assume that the associated~$2n$ critical points of~$f$ are evenly spaced on~$\mathbb{S}^1$.
The space~$\DiffSId$ of $\Id$-isotopic diffeomorphisms of the circle deformation retracts to~$\mathbb{S}^1$, i.e. the rotations of the circle. The subgroup~$\LevM$ of $\Id$-isotopic diffeomorphisms~$\phi$ preserving~$f$, that is $f\circ \phi=f$, is then (isomorphic to) the subgroup of rotations consisting of the~$2n$-th roots of unity that preserve the sequence of extremal values of~$f$. The result follows since the quotient of~$\mathbb S^1$ by a finite subgroup is again~$\mathbb S^1$. 
\end{proof}
When~$\M=[0,1]$ recall that Morse functions have prescribed values~$\fbM_0$ and~$\fbM_1$ on the boundary points~$0$ and~$1$. 
\begin{proposition}
\label{prop_interval_morse}
Assume~$\M=[0,1]$. Let~$D$ be a barcode and~$f\in \persmap^{-1}(D)$. Then~$\CompFiber$ is contractible.
\end{proposition}
\begin{proof}
From Proposition~\ref{theorem_fibration_orbit_stabiliser_diffeos}~$\CompFiber$ is homeomorphic to~$\DiffIId/ \LevM$. However~$\DiffIId$ deformation retracts on the identity diffeomorphism~$\mathrm{Id}_{[0,1]}$ by straight-line interpolations, and~$\LevM=\{\mathrm{Id}_{[0,1]}\}$.  
\end{proof}

Note that we could easily derive a similar statement for Morse functions on~$[0,1]$ without boundary conditions. In Section~\ref{appendix_interval_circle} we prove the analogues of Propositions~\ref{prop_circle_morse} and~\ref{prop_interval_morse} for continuous functions. The analogues for lower-star filtrations on the subdivided interval and circle have been proved by \citet{cyranka2018contractibility} and \citet{mischaikow2021persistent} respectively.
\begin{remark}
\label{remark_connected}
When~$\M=\M_1\sqcup \M_2$ has more than one connected component, the path component~$\CompFiber$ in the fiber over~$D=\persmap(f)$ can be retrieved as the product of the path components of the fibers over~$D_1:=\persmap(f_{|\M_1})$ and~$D_2:=\persmap(f_{|\M_2})$ containing the restrictions~$f_{|\M_1}$ and~$f_{|\M_2}$ respectively:
\[\CompFiber= \persmap_{f_{|\M_1}}^{-1}(D_1)\times \persmap_{f_{|\M_2}}^{-1}(D_2).\]
For this reason we focus our analysis to the interesting case where~$\M$ is connected.
\end{remark}
For the rest of the section we fix a compact connected surface~$\M$ and a function~$f$ with barcode~$D$, whose number of critical points of index~$1$ is denoted by~$c_1$. We make use of the analysis of the orbit~$\OrbId{}$ by~\citet{maksymenko2006homotopy}.
\begin{proposition}
\label{theorem_homotopy_type_surface_saddles_homotopy_groups}
Assume that~$c_1>0$. Then~$\pi_2(\CompFiber)=0$ and~$\pi_n(\CompFiber)=\pi_n(\M)$ for~$n\geq 3$. %
\end{proposition}
\begin{proof}
$\CompFiber=\OrbId$ by Theorem~\ref{theorem_fiber_equal_orbit}, and by~\citet[(2), Theorem~1.5]{maksymenko2006homotopy} we have~$\pi_2(\OrbId)=0$ and~$\pi_n(\OrbId)=\pi_n(\M)$ for~$n\geq 3$.
\end{proof}
\begin{remark}
\label{remark_fundamental_group}
From~\citet[(2), Theorem~1.5]{maksymenko2006homotopy} we can also derive a short exact sequence~$0\to \pi_1(\DiffMId)\oplus \Z^{k_f} \to \pi_1(\CompFiber)\to G \to 0$ where~$G$ is a finite group and the integer~$k_f\geq 0$ depends on the component~$\CompFiber$ in the fiber, on the number~$c_1$ of saddles and the surface~$\M$.
\end{remark}
\begin{proposition}
\label{theorem_homotopy_type_surface_no_saddles}
Assume that~$c_1=0$. Then the homotopy type of the fiber~$\CompFiber$ is classified as follows: %
\begin{center}
 \begin{tabular}{||c | c |c| c||} 
 \hline
 {\bf Surface} $\M$ & $\TwoSphere$ &  $\mathbb{S}^1\times \I$ &  $\mathbb{D}^2$ \\ [0.5ex] 
\hline 
{\bf Fiber} $\CompFiber$  & $\TwoSphere$ & $\{*\}$ & $\{*\}$\\ [0.5ex] 
  \hline
\end{tabular}
\end{center}
\end{proposition}
\begin{proof}
$\CompFiber=\OrbId$ by Theorem~\ref{theorem_fiber_equal_orbit}, and the homotopy type of~$\OrbId$ is computed by~\citet[Theorem~1.9]{maksymenko2006homotopy}.
\end{proof}
For instance the case where~$f:\TwoSphere\rightarrow \R$ has no saddle $(c_1=0)$ can be interpreted as follows: The fiber sequence~$\LevM \rightarrow \DiffSTwoId \rightarrow \CompFiber$ of Proposition~\ref{theorem_fibration_orbit_stabiliser_diffeos} can be identified up to homotopy with the standard fiber sequence $\mathbb{S}^1 \rightarrow \mathrm{SO}(3)\rightarrow  \TwoSphere$. This is because~$\DiffSTwoId$ deformation retracts to~$\mathrm{SO}(3)$ by the $2$-dimensional Smale conjecture (see ~\citet{smale1959diffeomorphisms}), and if without loss of generality we assume that~$f$ is the standard height function, then~$\LevM$ consists of those rotations fixing the poles, so~$\LevM$ is fixed by the retraction and~$\LevM\sim \mathbb{S}^1$.

\begin{proposition}
\label{theorem_homotopy_type_surface_saddles}
Assume that~$D$ has pairwise distinct bounded interval endpoints, and that~$c_1>0$. Then we have the following homotopy types for the fiber~$\CompFiber$: %
\begin{center}
 \begin{tabular}{||c | c |c| c| c| c||} 
 \hline
 {\bf Surface} $\M$ & $\TwoSphere$ & Projective Plane& Torus & $\mathbb{S}^1\times \I$ &  $\mathbb{D}^2$ \\ [0.5ex] 
\hline 
{\bf Fiber} $\CompFiber$  & $\SpeOrthogonal \times (\mathbb S^1)^{c_1-1}$ & $\SpeOrthogonal \times {(\mathbb S^1)}^{c_1-1}$ & $(\mathbb S^1)^{c_1+1}$ & $(\mathbb S^1)^{c_1}$& $(\mathbb S^1)^{c_1}$\\ [0.5ex] 
  \hline
\end{tabular}
\end{center}
When~$\M$ is obtained from the surfaces in the above tables by removing finitely many~$2$-disks, then~$\CompFiber\sim (\mathbb{S}^1)^{c_1-1}$. If~$\M$ is the Möbius strip, then~$\CompFiber\sim (\mathbb{S}^1)^{c_1}$. For other orientable surfaces~$\M$, we have~$\CompFiber\sim (\mathbb{S}^1)^{c_1+\chi(\M)}$. For the remaining non-orientable surfaces, we have~$\CompFiber\sim (\mathbb{S}^1)^{k_f}$ for some integer~$k_f\leq c_1+\chi(\M)$, unless~$\M$ is the Klein bottle in which case~$k_f\leq c_1+1$.
\end{proposition}
\begin{proof}
$\CompFiber=\OrbId$ by Theorem~\ref{theorem_fiber_equal_orbit}. Since~$D$ has distinct bounded interval endpoints,~$f$ has distinct critical points, and then the homotopy type of~$\OrbId$ is computed by~\citet[(2)\&(3), Theorem~1.5]{maksymenko2006homotopy}.
%
\end{proof}
\begin{remark}
\label{remark_saddle_points_barcode}
When~$\M$ has no boundary, $\bM=\emptyset$, the number~$c_1$ of saddles of Morse functions~$f$ in the fiber~$\persmap^{-1}(D)$ can be directly inferred from the barcode~$D$. Namely, if we denote by~$\IntervalsSurface$ the number of intervals in~$D$, then the quantity~$\IntervalsSurface-\beta_0-\beta_2$ counts (i) all the intervals~$(b,d)$ of~$D$ in degree~$1$, which correspond by their birth value~$b$ to saddle points of~$f$ whose attaching handle increases the~$1$-dimensional homology of the sub level-set~$f^{-1}((-\infty,b])$, and (ii) all the bounded intervals~$(b,d)$ of~$D$ in degree~$0$, which correspond by their death value~$d<\infty$ to saddle points of~$f$ whose attaching handle decreases the~$0$-dimensional homology of the sub level-set~$f^{-1}((-\infty,d])$. Hence~$c_1=\IntervalsSurface-\beta_0-\beta_2$. When~$\bM=\bigsqcup_j \bM_j \neq \emptyset$, we can partition the boundary components~$\bM_j$ into the sets~$\bM^{\mathrm{min}}$ (resp.~$\bM^{\mathrm{max}}$) of components~$\bM_j$ that are local minimum (resp. maximum) of one (hence any) function~$f$ in the component of~$\persmap^{-1}(D)$ at stake. Since~$\M$ is a surface each~$\bM_j$ is a circle, therefore if~$\bM_j\subseteq \bM^{\mathrm{min}}$, then it corresponds in the barcode~$D$ to the births of one interval in degree~$0$ and one interval in degree~$1$. Otherwise~$\bM_j\subseteq \bM^{\mathrm{max}}$ induces no topological change when entering the sub level-sets of~$f$. Consequently the correspondence between critical points and interval endpoints adapts and yields~$c_1=\IntervalsSurface-\beta_0-\beta_2-\#\bM^{\mathrm{min}}$.
\end{remark}
\begin{remark}
\label{remark_3_manifolds}
For manifolds~$\M$ of dimension~$3$ for which the Smale conjecture~$\DiffM \cong \mathrm{Isom}(\M)$ holds, e.g. the~$3$-sphere, lens spaces, prism and quaternionic manifolds (see~\cite{hong2012diffeomorphisms}), the homotopy type of~$\DiffMId$ is quite well-understood. For instance we have~$\mathcal{D}_{\mathrm{Id}}(\mathbb{S}^3)\cong \mathrm{SO}(4)$. However, to deduce the homotopy groups of~$\CompFiber$, we lack the understanding of less-studied topological properties of~$\LevM$.
\end{remark}
\section{Fiber of Persistent Homology for continuous maps on the circle and on the interval}
\label{appendix_interval_circle}

In this section the domain of the persistence map consists of continuous maps on the circle:
\[\persmap: \mathcal{C}^{0}(\mathbb S^1,\R) \longrightarrow \Barc^2. \]
Note that in the codomain we record the two barcodes with non-trivial homology, those in degree~$0$ and~$1$. In fact the second barcode contains a unique unbounded interval starting at the maximum of the function on the circle.

We fix a barcode~$D$ with finitely many intervals. When~$f=\mathrm{\underline{cst}}$ is constant it forms the fiber by itself over the trivial barcode~$D=\persmap(f)$ with only two infinite bars~$(\mathrm{\underline{cst}},+\infty)$, one in each degree~$0$ and~$1$. Other barcodes such that~$\persmap^{-1}(D)\neq \emptyset$ have one infinite interval~$(b_0,+\infty)$ in degree~$0$, one infinite interval~$(b_1,+\infty)$ with~$b_0<b_1$ in degree~$1$, finitely many bounded intervals in degree~$0$ with endpoints in~$[b_0,b_1]$, and no other intervals. In the rest of this section we assume that~$D$ is non-trivial and denote by~$(n-1)$, for some~$n\geq 1$, its number of bounded intervals in degree~$0$. 

Let~$\IncHomSpaceS$ be the space of orientation-preserving homeomorphisms of the circle, and~$\NonDecMapsS=\overbar{\IncHomSpaceS}$ be its closure in~$\mathcal{C}^{0}(\mathbb S^1,\R)$ in the compact-open topology. Given~$f\in \mathcal{C}^{0}(\mathbb S^1,\R)$ we have the pre-composition map~$\phi\in \NonDecMapsS \mapsto f\circ \phi \in \mathcal{C}^{0}(\mathbb S^1,\R)$; we denote by~$\LevM$ the stabiliser of~$f$ and by~$\OrbId$ its orbit. 
\begin{proposition}
\label{theorem_fiber_circular}
The fiber~$\persmap^{-1}(D)$ has finitely many path connected components. In each such component~$\CompFiberBis{}$ there exists some~$\SufInj:\mathbb S^1\rightarrow \R$ such that:
\[\CompFiberBis{}= \OrbIdSufInj{},\]
and then~$\CompFiberBis{}$ is homeomorphic to the quotient~$\NonDecMapsS/\LevMSufInj{}$, and in particular is homotopy equivalent to~$\mathbb S^1$.
\end{proposition}
Unlike the smooth case the component~$\CompFiberBis{}\subseteq\persmap^{-1}(D)$ in the fiber equals the orbit of a function only for a careful choice of function~$\SufInj$: the requirement will be that~$\SufInj$ is injective between its consecutive extrema. Nevertheless the fact that the pre-composition map induces a homeomorphism from~$\NonDecMapsS/\LevMSufInj{}$ to the orbit~$\OrbIdSufInj{}$ is reminiscent of the smooth case, and in fact with slightly more work it can be shown that it defines a~$\LevMSufInj{}$-principal bundle. We state without proof the analogous and simpler result for the unit interval~$[0,1]$, which works with or without fixed values on the boundary points~$0$ and~$1$. 
\begin{proposition}
\label{theorem_fiber_contractible}
For any finite barcode~$D$ the fiber~$\persmap^{-1}(D)\subseteq \mathcal{C}^{0}([0,1],\R)$ has finitely many path connected components, each of which is contractible.
\end{proposition}
Using a fixed orientation on~$\mathbb S^1$ and going around starting from the north pole we can order the~$n$ minima and~$n$ maxima of a non-constant~$f\in \mathcal{C}^{0}(\mathbb{S}^1,\R)$ into a sequence~$\Val(f)$ which we view as an element in~$\R^{2n}$:
\[\Val(f):=m_1(f)<M_1(f)>\cdots >m_n(f)<M_n(f).\]
Associated to this sequence we have the sequence of critical sets of~$f$:
\[\Seq(f): c_1(f),\,d_1(f),\cdots,\, c_{n}(f),\, d_n(f).\]
Explicitly, each~$c_i(f)$ (resp.~$d_i(f)$) is a connected component of~$f^{-1}(m_i(f))$ (resp. of~$f^{-1}(M_i(f))$).
\begin{proposition}
\label{prop_fiber_ordering_of_extrema}
Let~$f\in \persmap^{-1}(D)$. Then~$f$ has~$2n$ extrema, i.e.~$\Val(f)\in \R^{2n}$. In addition, let~$\CyclicGroup{}$ be the group of cyclic permutations on~$n$ elements, which acts on~$\R^{2n}$ by cyclically permuting the~$n$ pairs of entries. Then the connected component~$\CompFiberBis{}$ in the fiber containing~$f$ is made of functions~$g$ whose sequence of extrema is the same as that of~$f$ up to a different ordering, that is:
\begin{equation}
    \label{eq_fiber_ordering_of_extrema}
    \CompFiberBis{}=\big\{ g\in \mathcal{C}^{0}(\mathbb S^1, \R) \, | \, \Val(g)\in \CyclicGroup{}. \Val(f) \big \}
\end{equation}
\end{proposition}
We omit the proof of this elementary statement. So if~$\CompFiberBis{}$ is a component in the fiber, we can pick the following simple function~$\SufInj$ in~$\CompFiberBis{}$, whose critical sets and extrema are denoted by~$c_i,d_i,m_i,M_i$ for simplicity: the critical sets~$c_i$ and~$d_i$ are singletons arranged on the regular~$2n$-gon in~$\mathbb S^1$ and on each circular arc~$[c_{i},d_i]$,~$\SufInj$ restricts to the linear homeomorphism to~$[m_i,M_i]$.
\begin{proposition}
\label{prop_homeomorphism_precomposition}
Let~$\CompFiberBis{}\subseteq \persmap^{-1}(D)$ be a path component in the fiber. Then the pre-composition map~$\phi \mapsto \SufInj \circ \phi$ induces a homeomorphism from~$\NonDecMapsS/\LevMSufInj{}$ to~$\CompFiberBis{}$.
\end{proposition}
\begin{proof}
The map~$\phi\in \NonDecMapsS \mapsto \SufInj\circ \phi\in \CompFiberBis{}$ is well-defined, i.e.~$\persmap(f\circ\phi)=\persmap(f)=D$. This is because a homeomorphism~$\phi\in \IncHomSpaceS$ restricts to a homeomorphism between the sub level-sets of~$\SufInj\circ\phi$ and those of~$\SufInj$, hence it induces an isomorphism of persistent homology modules and the equality of barcodes~$\persmap(\SufInj\circ\phi)=\persmap(\SufInj)$, which holds as well for any~$\phi\in \NonDecMapsS= \overbar{\IncHomSpaceS}$ by continuity of~$\persmap$. 

Let~$f\in \CompFiberBis{}$. From Proposition~\ref{prop_fiber_ordering_of_extrema} there are cyclic permutations~$\pi\in \CyclicGroup{}$ such that~$\Val(f)=\pi.\Val(\SufInj)$. For each such permutation~$\pi$ there is a unique map~$\phi^{f,\pi}$ satisfying both~$\SufInj\circ \phi^{f,\pi}=f$ and~$\phi^{f,\pi}(c_i(f))=c_{\pi(i)}$ (and~$\phi^{f,\pi}(d_i(f))=d_{\pi(i)}$): It is defined on each circular arc~$[c_{i}(f),d_i(f)]$ by
\begin{equation}
    \label{eq_phi_f}
    \phi^{f,\pi}_{|[c_{i}(f),d_{i}(f)]}:=  [(\SufInj)_{|[c_{\pi(i)},d_{\pi(i)}]}]^{-1} \circ f_{|[c_{i}(f),d_i(f)]},
\end{equation}
and similarly on circular arcs~$[d_{i-1}(f),c_i(f)]$. In particular for~$f=\SufInj$ the set of such~$\phi^{f,\pi}$ equals the group~$\LevMSufInj{}$ of stabilisers. Therefore~$\phi\mapsto \SufInj\circ \phi$ descends to a continuous bijection from~$ \NonDecMapsS/\LevMSufInj{}$ to~$\CompFiberBis{}$. 

Finally we show that the inverse is continuous. Let~$f\in \CompFiberBis{}$ and~$\phi^{f,\pi}$ as in~\eqref{eq_phi_f}. Up to pre-composing~$f$ by a suitable homeomorphism the north pole does not belong to any extremal set~$c_i(f),d_i(f)$. Consequently, for~$g$ in a small neighborhood~$\Nbhood\subseteq \CompFiberBis{}$ around~$f$, we also have~$\Val(g)=\pi.\Val(\SufInj)$, hence we can define~$\phi^{g,\pi} \in \NonDecMapsS$ like in Eq.~\eqref{eq_phi_f} and then~$\SufInj\circ \phi^{g,\pi}=g$. Hence the map $g\in \Nbhood \longmapsto \phi^{g,\pi} \in \NonDecMapsS$ is a local section, whose continuity is a consequence of the fact that on each circular arc~$[c_i,d_i]$ the linear restriction~$(\SufInj)_{|[c_i,d_i]}$ and its inverse are Lipschitz, and of the fact that the maximal distance from points in the critical sets~$c_i(g),d_i(g)$ to the critical sets~$c_i(f),d_i(f)$ of~$f$ can be continuously tracked in a sufficiently small neighborhood~$\Nbhood{}$ of~$f$, see~Fig.~\ref{fig:tracking_extrema}. The technical details are omitted.
\end{proof}
\begin{figure}[htbp]
\centering
\resizebox{ .6\textwidth}{!}{
\begin{tikzpicture}[
  declare function={
    f(\x)= (\x < 0) * ({\x*\x*\x*\x - \x*\x + 1})   +
              and(\x >= 0, \x < .3) * (1)     +
              (\x >= .3) * ({(\x-.3)*(\x-.3)*(\x-.3)*(\x-.3) - (\x-.3)*(\x-.3) + 1})
   ;
  },
  ]
\draw[black, line width = 0.05mm]   plot[smooth,domain=-1.05:.6] (\x,{(\x < 0) * (\x*\x*\x*\x - \x*\x + 1)   +
              and(\x >= 0, \x < .3) * (1)     +
              (\x >= .3) * ((\x-.3)*(\x-.3)*(\x-.3)*(\x-.3) - (\x-.3)*(\x-.3) + 1)});

\draw[line width = 0.05, dash pattern=on 1pt off 1pt]   plot[smooth,domain=-1.05:.6] (\x,{(\x < 0) * (\x*\x*\x*\x - \x*\x + 1.05)   +
              and(\x >= 0, \x < .3) * (1.05)     +
              (\x >= .3) * ((\x-.3)*(\x-.3)*(\x-.3)*(\x-.3) - (\x-.3)*(\x-.3) + 1.05)});

\draw[line width = 0.05, dash pattern=on 1pt off 1pt]   plot[smooth,domain=-1.05:.6] (\x,{(\x < 0) * (\x*\x*\x*\x - \x*\x + .95)   +
              and(\x >= 0, \x < .3) * (.95)     +
              (\x >= .3) * ((\x-.3)*(\x-.3)*(\x-.3)*(\x-.3) - (\x-.3)*(\x-.3) + .95)});

\draw[line width = 0.05, dash pattern=on \pgflinewidth off 0.5pt] (-1.05,.75) -- (.6,.75);
\node[scale = 0.15] at (-1,.78) {$m_i$};

\draw[line width = 0.05, dash pattern=on \pgflinewidth off 0.5pt] (-.85,.75) -- (-.85,.5);
\draw[line width = 0.05, dash pattern=on \pgflinewidth off 0.5pt] (-.53,.75) -- (-.53,.5);
\draw[line width = 0.05] (-.85,.5) -- (-.53,.5);
\node[scale = 0.15] at (-.707,.53) {${U}_i$};

\draw[line width = 0.05, dash pattern=on \pgflinewidth off 0.5pt] (-1.05,1) -- (.6,1);
\node[scale = 0.15] at (-.4,1.03) {$M_i$};

\draw[line width = 0.05, dash pattern=on \pgflinewidth off 0.5pt] (-.23,1) -- (-.23,.5);
\draw[line width = 0.05, dash pattern=on \pgflinewidth off 0.5pt] (.53,1) -- (.53,.5);
\draw[line width = 0.05] (-.23,.5) -- (.53,.5);
\node[scale = 0.15] at (.15,.53) {$V_i$};

\draw[line width = 0.05, dash pattern=on \pgflinewidth off 0.5pt] (-.707,.75) -- (-.707,1.2);
\node at (-.707,1.2)[circle,fill, line width = .05mm, inner sep=.15pt, minimum size = 0pt]{};
\node[scale = 0.15] at (-.707,1.24) {$c_i(f)$};

\draw[line width = 0.05, dash pattern=on \pgflinewidth off 0.5pt] (0,1) -- (0,1.2);
\draw[line width = 0.05, dash pattern=on \pgflinewidth off 0.5pt] (.3,1) -- (.3,1.2);
\draw[line width = 0.05] (0,1.2) -- (.3,1.2);
\node[scale = 0.15] at (.15,1.24) {$d_i(f)$};

\end{tikzpicture}

}
\caption{A piece of a continuous function $f:\mathbb{S}^1\rightarrow \R$ and a small neighborhood~$\Nbhood$ indicated by dashed curves. Any function in~$\persmap^{-1}(D)$ between the dashed curves must have a critical value in each~$U_i$ and~$V_i$, provided the band between the dashed curves is thin enough to separate critical values. If $g$ is such a function then these must be the only critical values. Then the critical value of~$g$ in~$U_i$ must be~$m_i$, in~$V_i$ must be~$M_i$, and so on.}

\label{fig:tracking_extrema}

\end{figure}
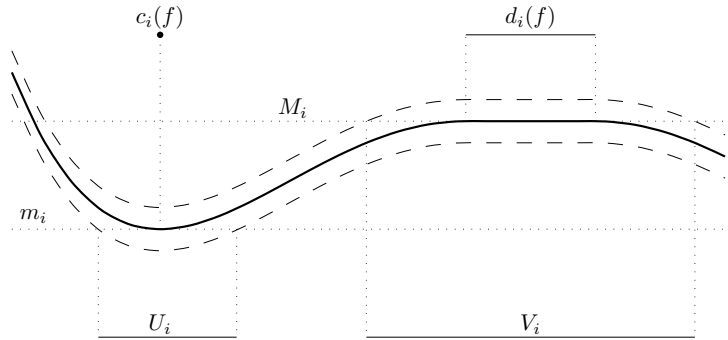

\begin{proof}[Proof of~Proposition~\ref{theorem_fiber_circular}]
From Proposition~\ref{prop_homeomorphism_precomposition} the pre-composition map~$\phi \mapsto \SufInj\circ \phi$ induces a homeomorphism from~$\NonDecMapsS/ \LevMSufInj{}$ to the path connected component~$\CompFiberBis{}$. Besides it is well-known that~$\NonDecMapsS$ deformation retracts to the group~$\mathrm{SO}(2)\cong \mathbb{S}^1$ of orientation preserving rotations.\footnote{For instance the deformation retract of~$\IncHomSpaceS$ of~\citet[Theorem~1.1.2]{hamstrom1974homotopy} extends to~$\NonDecMapsS$.} Recall that~$\SufInj$ is a piece-wise linear interpolation between extremal values arranged on a regular~$2n$-gon, therefore its stabiliser~$\LevMSufInj{}$ is a finite subgroup of~$\mathrm{SO}(2)$ which is preserved under the deformation retraction. Hence~$\CompFiberBis{}$ is homotopy equivalent to the quotient of~$\mathrm{SO}(2)\cong \mathbb{S}^1$ by a finite subgroup, so it is in fact homotopy equivalent to~$\mathbb{S}^1$.  
\end{proof}

\subsection*{Conflict of Interest} On behalf of all authors, the corresponding author states that there is no conflict of interest. 

\bibliographystyle{plainnat}
\bibliography{reference}
\end{document}